\documentclass[11pt]{amsart}
\usepackage[top = 1in, bottom = 1in, left = 1.5in, right=1.5in,
marginparwidth=1.3in]{geometry}
\usepackage{setspace}
\usepackage{amsmath, amsfonts, amsthm, amstext, xspace}
\usepackage[mathscr]{euscript}
\usepackage{amssymb,latexsym}
\usepackage{comment}
\usepackage[dvipsnames,svgnames]{xcolor}
\usepackage{hyperref}
\usepackage[capitalise]{cleveref}
\usepackage{adjustbox}

\usepackage[utf8]{inputenc}
\usepackage[T1]{fontenc}

\definecolor{seagreen}{RGB}{46,139,87}
\definecolor{maroon}{RGB}{128,0,0}
\definecolor{darkviolet}{RGB}{148,0,211}
\definecolor{twelve}{RGB}{100,100,170}
\definecolor{thirteen}{RGB}{100,150,50}
\definecolor{fourteen}{RGB}{200,0,0}
\definecolor{fifteen}{RGB}{0,200,0}
\definecolor{sixteen}{RGB}{0,0,200}
\definecolor{seventeen}{RGB}{200,0,200}
\definecolor{eighteen}{RGB}{0,200,200}

\usepackage{centernot}
\usepackage{longtable}
\usepackage{mathtools}
\usepackage{stmaryrd}
\makeatletter
\newcommand{\xMapsto}[2][]{\ext@arrow 0599{\Mapstofill@}{#1}{#2}}
\def\Mapstofill@{\arrowfill@{\Mapstochar\Relbar}\Relbar\Rightarrow}
\makeatother

\usepackage[nolabel]{showlabels} 

\usepackage[all]{xy}
\newtheorem{thm}[equation]{Theorem}
\newtheorem*{theorem*}{Theorem}
\newtheorem*{conjecture*}{Conjecture}
\newtheorem*{corollary*}{Corollary}
\newtheorem{lemma}[equation]{Lemma}
\newtheorem{theorem}[equation]{Theorem}
\newtheorem{corollary}[equation]{Corollary}

\newtheorem{proposition}[equation]{Proposition}

\theoremstyle{definition}
\newtheorem{definition}[equation]{Definition}
\newtheorem{example}[equation]{Example}
\newtheorem{remark}[equation]{Remark}

\newtheorem{construction}[equation]{Construction}

\numberwithin{equation}{subsection} 

\newtheorem{thmx}{Theorem}

\def\l{\mathbb{L}}

\def\r{\mathbb{R}}

\def\z{\mathbb{Z}}

\def\cc{\mathcal{C}}

\def\ce{\mathcal{E}}

\def\co{\mathcal{O}}

\def\uI{\underline{I}}
\def\uM{\underline{M}}
\def\uP{\underline{P}}
\def\uR{\underline{R}}

\def\uA{\underline{A}}
\def\uC{\underline{C}}
\def\uH{\underline{H}}
\def\uM{\underline{M}}

\def\Ab{\mathcal{A}b}

\def\Tor{\operatorname{Tor}}

\def\cofib{\operatorname{cofib}}

\def\ker{\operatorname{ker}}
\def\coker{\operatorname{coker}}

\def\mfz{\underline{\mathbb{Z}}}
\newcommand{\Conf}{\operatorname{Conf}}
\newcommand{\Aut}{\operatorname{Aut}}

\def\upi{\underline{\pi}}
\def\Emb{\operatorname{Emb}}

\def\res{\operatorname{res}}

\usepackage{marginnote}
\usepackage{soul}
\newcommand{\sm}{\wedge}

\newcommand{\attop}[1]{{\let\textstyle\scriptstyle\let\scriptstyle\scriptscriptstyle\substack{#1}}}
\renewcommand{\atop}[1]{{\let\scriptstyle\textstyle\let\scriptscriptstyle\scriptstyle\substack{#1}}}
\newcommand{\x}{\times}
\renewcommand{\st}{\hspace{2pt} : \hspace{2pt}}

\renewcommand{\l}{\overset}
\newcommand{\too}[1]{\l{#1}\to}
\newcommand{\hteq}{\simeq}
\newcommand{\ints}{\cap}
\newcommand{\isom}{\cong}
\newcommand{\Z}{\mathbb{Z}}

\newcommand{\til}[1]{\widetilde{#1}}

\makeatletter
\newcommand{\switchmargin}{
\if@reversemargin
\normalmarginpar
\else
\reversemarginpar
\fi
}
\makeatother

\newcommand{\highlighteva}[1]{\ifmmode{\text{\sethlcolor{llgray}\hl{$#1$}}}\else{\sethlcolor{llteal}\hl{#1}}\fi}

\definecolor{llteal}{RGB}{198,232,227}
\definecolor{llred}{RGB}{237,228,228}
\definecolor{llgray}{RGB}{230,230,230}
\definecolor{maroon}{RGB}{150,0,0}
\definecolor{orange}{RGB}{255,165,0}

\newcommand{\highlight}[1]{\ifmmode{\text{\sethlcolor{llgray}\hl{$#1$}}}\else{\sethlcolor{llred}\hl{#1}}\fi}

\newcommand{\ZZ}{\mathbb Z}
\newcommand{\sma}{\wedge}
\newcommand{\sus}{\Sigma}

\usepackage{amssymb}
\usepackage{tikz-cd}

\author{Eva Belmont}\address{Case Western Reserve University}\email{eva.belmont@case.edu}
\author{J.D. Quigley}\address{University of Virginia}\email{mbp6pj@virginia.edu}
\author{Chase Vogeli}\address{Cornell University}\email{cpv29@cornell.edu}
\title{Bredon homological stability for configuration spaces of $G$-manifolds}

\begin{document}

\begin{abstract}
McDuff and Segal proved that unordered configuration spaces of open manifolds satisfy
homological stability: there is a stabilization map $\sigma: C_n(M)\to
C_{n+1}(M)$ which is an isomorphism on $H_d(-;\Z)$ for $n\gg d$.
For a finite group $G$ and an open $G$-manifold $M$, under some hypotheses we
define a family of equivariant stabilization maps $\sigma_{G/H}:C_n(M)\to
C_{n+|G/H|}(M)$ for $H\leq G$. In general, these do not induce
stability for Bredon homology, the equivariant analogue of singular homology. Instead, we show that each $\sigma_{G/H}$ induces
isomorphisms on the ordinary homology of the fixed points of $C_n(M)$,
and if the group is Dedekind (e.g. abelian), we obtain the following Bredon
homological stability statement: $H^G_d(\bigsqcup_{n\geq 0}C_n(M))$ is finitely
generated over $\Z[\sigma_{G/H} : H\leq G]$. This reduces to the classical
statement when $G=e$.
\end{abstract}

\maketitle

\tableofcontents

\section{Introduction}

Let $X$ be a topological space and let $C_k(X)$ denote the configuration space of $k$ unordered points in $X$. If $X$ is the interior of a manifold with nonempty boundary, then there are stabilization maps
\begin{equation}\label{Eqn:StabIntro}
\sigma: C_k(X) \to C_{k+1}(X)
\end{equation}
defined by ``adding a point near the boundary" (cf. \cref{SS:Conf}). In the 1970's, McDuff \cite{McD75} and Segal \cite{Seg79} observed that these stabilization maps induce isomorphisms in a range of integral homology groups:

\begin{thm}[McDuff--Segal, Stability for Configuration Spaces (Strong Form)]\label{Thm:MS}
Assume $M$ is the interior of a manifold with nonempty boundary. The stabilization map \eqref{Eqn:StabIntro} induces an isomorphism
$$\sigma_*: H_d (C_k(M)) \to H_d(C_{k+1}(M))$$
for $d \leq k/2$. 
\end{thm}

\subsection{Statement of main results}

In this paper, we investigate an equivariant analogue of the McDuff--Segal theorem where the manifold $M$ is replaced by a $G$-manifold (where $G$ is a finite abelian group) and singular homology is replaced by Bredon homology. 

The Bredon homology of a $G$-space $X$ with coefficients in the constant Mackey functor $\mfz$, written $H^G_*(X; \mfz)$, is an equivariant analogue of singular homology with integer coefficients. We recall this definition of Bredon homology, as well as its representation by a $G$-spectrum, in \cref{SS:Bredon}.

If $X$ is a $G$-space, then the space of unordered configurations of $k$ points in $X$
$$C_k(X) := \{(x_1,\ldots,x_k) \in X^k: x_i \neq x_j \text{ for all } i\neq j\}/\Sigma_k$$
is also a $G$-space with $G$ acting diagonally. If $X$ is a $G$-manifold which is the interior of a $G$-manifold with nonempty boundary containing a fixed point, then we can define $G$-equivariant stabilization maps
\begin{equation}\label{Eqn:FPGIntro}
\sigma_{G/G}: C_k(X) \to C_{k+1}(X)
\end{equation}
by ``adding a fixed point near the boundary'' (see \cref{SS:GConf}). In this situation, we say $X$ is \emph{$G$-stabilizable} (\cref{Def:HStab}). We can then ask if these maps induce isomorphisms in Bredon homology for $k$ sufficiently large. Surprisingly, this is not the case:

\begin{thmx}[Failure of Strong Equivariant Form of McDuff--Segal, Proposition \ref{Prop:H0Unstable}]\label{MT:BredonFails}
Let $G=C_p$ and let $M$ be a $G$-manifold which is $G$-stabilizable. The map
\begin{equation}\label{Eqn:GStabIntro}
(\sigma_{G/G})_*: H_0^G(C_n(M); \mfz) \to H_0^G(C_{n+1}(M); \mfz)
\end{equation}
is not surjective for any $n \geq 1$.
\end{thmx}

Although the strong form of the McDuff--Segal Theorem fails equivariantly, there is a weaker form which we can extend to the equivariant setting. Let
$$C(X) := \bigsqcup_{k \geq 0} C_k(X)$$
denote the space of all unordered configurations of finitely many points in $X$ equipped with the disjoint union topology. The nonequivariant stabilization maps \eqref{Eqn:StabIntro} give rise to a self-map
\begin{equation}\label{Eqn:BigStabIntro}
\sigma: C(X) \to C(X)
\end{equation}
which allows us to state the following corollary of the McDuff--Segal Theorem:

\begin{thm}[McDuff--Segal, Nonequivariant Stability (Finite Generation Form)]
Assume $M$ is an open, connected manifold with $\dim(M)\geq2$. Moreover, suppose that $H_d(M)$ is finitely generated as an abelian group for all $d\geq 0$. Then the homology group $H_d(C(M))$ is finitely generated as a $\z[\sigma_*]$-module for all $d\geq 0$. 
\end{thm}

\begin{remark}
The assumption that $H_d(M)$ is finitely generated ensures that $H_d(C_k(M))$ is finitely generated for all $k \geq 0$ (\cref{Prop:MtoCnM}), so the finite generation form of nonequivariant stability is equivalent (cf. \cref{Prop:AlgStability}) to saying that the map \eqref{Eqn:StabIntro} is an isomorphism for all $k \gg 0$. 

We note that the philosophy of encoding stability as a finite generation condition is not new. It has become increasingly prevalent in the literature on homological stability (e.g., \cite[Theorem 1.1]{ADCK20}) and representation stability (e.g. \cite[Theorem 6.2.1]{CEF-FI}).
\end{remark}

The obvious finite generation form of equivariant stability fails since the equivariant stabilization map \eqref{Eqn:GStabIntro} does not induce isomorphisms for all $k \gg 0$. In other words, $H_d^G(C(M); \mfz)$ is \emph{not} finitely generated as a $\z[(\sigma_{G/G})_*]$-module. 

However, there is a naturally occurring polynomial ring over which $H_d^G(C(M); \mfz)$ is a finitely generated module. If $X$ is a $G$-manifold which is the interior of a $G$-manifold with boundary containing an orbit of the form $G/H$, then we can define a new equivariant stabilization map
\begin{equation}\label{Eqn:GHStabIntro}
\sigma_{G/H}: C_k(M) \to C_{k+|G/H|}(M)
\end{equation}
by ``adding an \emph{orbit} of type $G/H$ near the boundary.'' In this situation, we say $X$ is \emph{$H$-stabilizable} (\cref{Def:HStab}). 

We adopt the philosophy that since orbits are equivariant analogues of points, the equivariant analogue of stabilization maps should be the family of stabilization maps $\sigma_{G/H}$. As such, we replace the ring $\z[\sigma_*]$ by the ring 
$$P_G=\z[(\sigma_{G/H})_*: (H) \leq G]$$
of all equivariant stabilization maps. By considering all possible equivariant stabilization maps, we obtain our stability result.


\begin{thmx}[{Bredon Homological Stability, \cref{thm:B}}]\label{MT:Stability}
Let $G$ be an abelian group and $M$ a $G$-manifold that is the interior of a
compact $G$-manifold with boundary. Assume that for all $H\leq G$, $M$ is
$H$-stabilizable and $M^H$ is connected.
Then $H_d^H(C(M);\underline{\Z})$ is finitely generated over $P_G$ for all $H\leq G$.
\end{thmx}
For example, these hypotheses are satisfied in the case where $M$ is a sum of
copies of the regular representation (see Example \ref{ex:rho}).

\begin{remark}
See Lemma \ref{lem:RK-to-RG} for the main reason for the abelian hypothesis, and
Remark \ref{rmk:abelian} for a mild relaxation of this hypothesis.
Group-theoretic hypotheses are also used in Lemma
\ref{lem:hypothesis-translation} and \ref{lem:fg-homology}.
\end{remark}

\begin{remark}
We discuss two natural extensions of \cref{MT:Stability} in \cref{Sec:Gen}: in \cref{Thm:ROG}, we show that stability holds for $RO(G)$-graded Bredon homology, and in \cref{Thm:MF}, we show that stability holds for Mackey functor-valued Bredon homology. The $RO(G)$-graded Bredon homology of \emph{ordered} configuration spaces of $G$-manifolds will also be studied in forthcoming work of Dugger and Hazel \cite{DH23}. 
\end{remark}

\subsection{Motivation}

Before discussing the proofs of \cref{MT:BredonFails} and \cref{MT:Stability}, we discuss some motivating ideas and directions for future work.

\subsubsection{Equivariant loop spaces and Dyer--Lashof operations}

Nonequivariantly, May's recognition principle \cite{May72} implies that $n$-fold loop spaces are equivalent to algebras over the $E_n$-operad $\ce_n$ whose $k$-th space is $C_k(\r^n)$. Using Cohen's computation of $H_*(C_k(\r^n))$ in \cite[Ch. III]{CLM76}, one can then define Dyer--Lashof operations on the mod $p$ homology of any $n$-fold loop space. 

In \cite{GM17}, Guillou and May proved an equivariant analogue of the recognition principle: if $V$ is a real orthogonal $G$-representation, then $V$-fold loop spaces are equivalent to algebras over the $E_V$-operad $\ce_V$ whose $k$-th space is $C_k(V)$. They remark that in contrast with the nonequivariant situation, very little is known about $H_d^G(C_k(V))$. \cref{MT:Stability} gives the first step toward a more systematic analysis of these Bredon homology groups which could shed light on equivariant Dyer--Lashof operations (cf. \cite{Wil19} for $G=C_2$). We verify that the hypotheses of \cref{MT:Stability} are satisfied in the case of sums of regular representations in \cref{ex:rho}.

\subsubsection{Equivariant diffeomorphism groups of smooth $G$-manifolds}

Configuration spaces of ordinary manifolds appear in the study of diffeomorphism groups of smooth manifolds via the ``decoupling theorems" of B{\"o}digheimer--Tillmann \cite{BT01} and Bonatto \cite{Bon22}. It would be interesting to see if the configuration spaces of $G$-manifolds arise analogously in the equivariant diffeomorphism groups of smooth $G$-manifolds.

\subsection{Overview of key ideas}

We now summarize the key ideas going into the proofs of \cref{MT:BredonFails} and \cref{MT:Stability}. 

\subsubsection{Equivariant stabilization maps}

A map between $G$-spaces induces a map on Bredon homology only if it is $G$-equivariant, so we must be careful about how we add points. One natural option is to add a \emph{fixed} point near infinity (provided our space has such a fixed point), which yields the equivariant stabilization map \eqref{Eqn:FPGIntro} mentioned above with underlying map \eqref{Eqn:StabIntro}. However, \cref{MT:BredonFails} shows that this map does not produce isomorphisms in Bredon homology after iteration. 

One of our key observations is that under suitable hypotheses on the manifold $M$, there is a stabilization map \eqref{Eqn:GStabIntro} for each subgroup $H \leq G$. \cref{MT:Stability} says that the Bredon homology stabilizes when all of these equivariant stabilization maps are taken into account simultaneously. 

\begin{remark}
The stabilization map corresponding to $H \leq G$ inserts an orbit of the form $G/H$, i.e., inserts $|G/H|$ points at once. Other stabilization maps which add multiple points at once have appeared in the context of higher order representation stability \cite{MW19}.
\end{remark}

\subsubsection{Formulating stability in terms of finite generation}

While none of the orbit stabilization maps \eqref{Eqn:GStabIntro} yield Bredon homological stability for a fixed $H$ (\cref{MT:BredonFails}), the collection of maps $\{ \sigma_{G/H} : (H) \leq G\}$ yield a form of stability (\cref{MT:Stability}). To make this precise, we translate homological stability from a statement about maps in homology to a statement about finite generation over a polynomial ring (\cref{Prop:AlgStability}). 

\subsubsection{Reduction from Bredon homology to singular homology of fixed points}
The analogue of cellular homology in the context of $G$-spaces is \emph{Bredon homology},
$H^G_d(-)$ (cf. \cite[XIII.4.1]{May96}, Section \ref{SS:Bredon}). The Bredon
homology of a $G$-space is not simply the homology of its fixed points, but in
Section \ref{Sec:Eqvt} we use equivariant stable homotopy theory to relate these
notions. The main lemma enabling this comparison is
\cref{Lemma:Bredon-to-PhiG}, together with statements in Section \ref{SS:GEM} relating the geometric localization of Bredon homology to singular homology.

\subsubsection{Expressing fixed points in terms of nonequivariant configuration spaces}

We prove that the fixed points $C_n(M)^G$ satisfy homological stability by reducing to \cref{Thm:MS}. To do so, we express $C_n(M)^G$ in terms of configuration spaces of other manifolds using the notion of \emph{unordered $S$-configurations}. We note that we are not the first to consider equivariant configuration spaces (cf. \cite{RS00}), but we use a new convention here.

Let $X$ be a $G$-space and let $S$ be a finite $G$-set. We define the \emph{space of unordered $S$-configurations in $X$} (\cref{Def:CGS}) by
$$C_S^G(X) := \Emb^G(S,X)/\Aut(S).$$
When $M$ is $H$-stabilizable, the equivariant stabilization map \eqref{Eqn:GHStabIntro} restricts to a stabilization map 
$$\sigma_{G/H} : C_S^G(M) \to C_{S + G/H}^G(M),$$
where $S + G/H$ denotes the disjoint union of $S$ and $G/H$ as $G$-sets.

\begin{example}
If $S$ is a trivial $G$-set of cardinality $k$, then we have that
$$C_{S}^G(X) = C_k(X^G)$$
is the space of unordered configurations of $k$ points in the $G$-fixed points of $X$. If $X$ is a $G$-stabilizable $G$-manifold, the stabilization map
$$\sigma_{G/G} : C_{S}^G(X) \to C_{S + [G/G]}^G(X)$$
is the classical stabilization map
$$\sigma: C_k(X^G) \to C_{k+1}(X^G).$$
\end{example}

Passing through the $S$-configuration spaces $C_S^G(X)$ allows us to identify the fixed points $C_n(X)^G$ in terms of ordinary configuration spaces:

\begin{thmx}[{\cref{Prop:FixptConfDecomp} and \cref{Cor:CGSDecomp}}]\label{MT:CtoCSG}
\
\begin{enumerate}
\item For a $G$-space $X$ and $n\geq 0$, there is a disjoint union decomposition
$$ C_n(X)^G \cong \coprod_{|S|=n} C^G_S(X), $$
where the disjoint union runs over isomorphism classes of finite $G$-sets $S$ with cardinality $n$.

\item Let $S$ be the finite $G$-set
$$S = \bigsqcup_{(H)} k_{(H)}[G/H],$$
where $k_{(H)} \geq 0$ for each conjugacy class of subgroups $(H)$. Then
$$C_S^G(X) \cong \prod_{(H)} C_{k_{(H)}}(X_{(H)}/G),$$
where $X_{(H)}$ is the subspace of $X$ consisting of points with stabilizer conjugate to $(H)$ (\cref{Def:X_(H)}). 
\end{enumerate}
\end{thmx}

We apply part (1) of \cref{MT:CtoCSG} to prove \cref{MT:BredonFails}; the essential point is that when $G$ is nontrivial, the number of summands appearing in $C_n(X)^G$ increases as $n$ increases, so the stabilization map cannot be surjective on homology. We apply parts (1) and (2) to prove \cref{MT:CSGStability}.

\subsubsection{Stability for $S$-configurations and Bredon homology}

Since they can be expressed in terms of ordinary configuration spaces, the spaces of $S$-configurations often satisfy an analogue of the McDuff--Segal theorem:

\begin{thmx}[{\cref{Thm:CGSHomStab}}]\label{MT:CSGStability}
Let $M$ be a $G$-manifold and $S$ a finite $G$-set containing $k$
copies of the orbit
$[G/H]$ for a conjugacy class of subgroups $(H)$. If $M$ is $H$-stabilizable and
$M_{(H)}/G$ is connected, then the $(H)$-stabilization map of \eqref{eq:sigmaGH} induces an isomorphism
$$ (\sigma_{G/H})_*: H_d(C^G_S(M)) \to H_d(C^G_{S+[G/H]}(M)) $$
in integral homology in degrees $d\leq k/2$.
\end{thmx}

As explained above, \cref{MT:Stability} follows from stability for the fixed points $C_n(M)^G$. \cref{MT:CtoCSG} allows us to express these fixed points in terms of $S$-configurations, and \cref{MT:CSGStability} states that the equivariant stabilization map $\sigma_{G/H}$ from \eqref{Eqn:GHStabIntro} induces an isomorphism in the homology of $C_S^G(M)$ when $S$ contains sufficiently many copies of $[G/H]$. In particular, once $S$ contains sufficiently many copies of $[G/K]$ for \emph{any} conjugacy class of subgroups $(K)$, there is some equivariant stabilization map which induces an isomorphism on homology. Roughly speaking, \cref{MT:Stability} then follows from the fact that there are finitely many finite $G$-sets $S$ for which none of the equivariant stabilization maps induce isomorphisms in homology. 

\subsection{Outline}

In \cref{Sec:Eqvt}, we recall the equivariant homology theory (Bredon homology) we work with throughout the paper. We also recall the isotropy separation sequence and prove some key results (\cref{Lemma:Bredon-to-PhiG}, \cref{Prop:Green}) needed to reduce from the equivariant to nonequivariant setting. 

In \cref{Sec:Conf}, we recall configuration spaces, stabilization maps, and their equivariant analogues. We then introduce the more refined spaces of $S$-configurations and use them to prove \cref{MT:CtoCSG} which identifies the fixed points of the configuration spaces of $G$-manifolds with (sums of products of) configuration spaces in nonequivariant manifolds. 

In \cref{Sec:Stab}, we prove \cref{MT:BredonFails},  \cref{MT:Stability}, and \cref{MT:CSGStability}, using the results of the previous section to reduce to classical homological stability results.

In \cref{Sec:Gen}, we extend our results to the $RO(G)$-graded and Mackey functor-valued settings. 

\subsection{Conventions}

Throughout, $G$ denotes a finite group. If $M$ is an abelian group, then $\uM$ denotes the constant Mackey functor on $M$. 
In this work, ``manifold'' means ``smooth manifold.''

\subsection{Acknowledgments}

The authors thank Dan Dugger, Jeremy Hahn, Ben Knudsen, Irakli Patchkoria, Oscar Randal-Williams, Bridget Schreiner, Inna Zakharevich, and Mingcong Zeng for helpful discussions.
The first author was partially supported by NSF grant DMS-2204357.
The second author was partially supported by NSF grants DMS-2039316 amd DMS-2314082, as well as an AMS-Simons Travel Grant. The second author also thanks the Max Planck Institute in Bonn for providing a wonderful working environment and financial support during the beginning of this project.

\section{Equivariant homotopy theory}\label{Sec:Eqvt}

In this section, we discuss the ideas from equivariant homotopy theory which we will need in the sequel. \cref{SS:Bredon} contains a rapid recollection of Bredon homology with coefficients in a Mackey functor, which is the central equivariant homology theory we use throughout this work. \cref{SS:GFP} and \cref{SS:PHO} recall two standard constructions from equivariant stable homotopy theory: geometric fixed points and proper homotopy orbits. These auxiliary constructions appear in the isotropy separation sequence, which we recall in \cref{SS:ISS}. Our main technical result, \cref{Lemma:Bredon-to-PhiG}, applies the isotropy separation sequence to pass from Bredon homology to a simpler nonequivariant homology theory. This simpler theory is identified as a sum of singular homology theories in \cref{SS:GEM}. 

\subsection{Bredon homology}\label{SS:Bredon}

Bredon homology associates to a $G$-space a sequence of Mackey functors (see \cite[\S IX.4]{May96}). As in nonequivariant homotopy theory, Bredon homology with coefficients in a Mackey functor $\uM$ is representable in the equivariant stable homotopy category. By \cite[Theorem XIII.4.1]{May96}, there is an Eilenberg--MacLane $G$-spectrum $H\uM$ with the property that its homotopy Mackey functors satisfy $\upi_0(H\uM) \cong \uM$ and $\upi_i(H\uM)=0$ for $i>0$. If $X$ is a $G$-space, then we have
$$\uH_*^{(-)}(X;\uM) \cong \upi_*(\Sigma^\infty_G X \wedge H\uM),$$
where $\Sigma^\infty_G X$ is the \emph{suspension $G$-spectrum of $X$}, and $- \wedge - $ is the smash product of $G$-spectra. 

The Bredon homology of a $G$-CW complex can equivalently be defined using an equivariant analogue of the cellular chain complex \cite[\S I.4]{May96}. A \emph{coefficient system} is a contravariant functor $\co_G \to \Ab$, where $\co_G$ is the \emph{orbit category} of $G$. If $X$ is a $G$-CW complex, we may define a coefficient system by
$$\uC_n(X)(G/H) := \uH_n(X^n,X^{n-1}; \z)(G/H) := H_n((X^n)^H,(X^{n-1})^H; \z).$$
The connecting homomorphisms for the triples $((X^n)^H, (X^{n-1})^H, (X^{n-2})^H)$ define a map of coefficient systems
$$d: \uC_n(X) \to \uC_{n-1}(X).$$
If $\uM$ is a \emph{Mackey functor}, we may define cellular chains with coefficients in $\uM$ by
$$C_n^G(X;\uM) := \uC_n(X) \otimes \uM := \bigoplus_{G/H \in \co_G} \left( \uC_n(X)(G/H) \otimes \uM(G/H) \right)/\sim,$$
where $f^*c\otimes m \sim c \otimes f_*m$ for a map $f: G/H \to G/K \in \co_G$ and elements $c \in \uC_n(X)(G/H)$ and $m \in \uM(G/H)$. This is a chain complex of abelian groups, with differential given by $\partial = d \otimes 1$. By forgetting structure, we may analogously define 
$$C_n^H(X;\uM) := \uC_n(i^*_HX) \otimes i^*_H \uM$$
for each subgroup $H < G$; these assemble into a chain complex of coefficient systems whose homology
$$\uH_*^{(-)}(X;\uM) := H_*(C_n^{(-)}(X;\uM))$$
is the \emph{Bredon homology of $X$ with coefficients in $\uM$}. 

For much of the paper, we will be concerned with $\Z$-graded Bredon homology \emph{groups}. By this, we mean the $G/G$-levels $\uH_*^{G}(X;\uM)$ of the Bredon homology Mackey functors constructed here. We return to the level of generality of Bredon homology Mackey functors in \cref{Sec:Gen}. It is also possible to extend Bredon homology with coefficients in a Mackey functor from a $\Z$-graded collection of groups to a $RO(G)$-graded collection of groups, where $RO(G)$ denotes the representation ring of $G$ \cite[\S IX.5]{May96}. We review this extension in \cref{Sec:Gen} as well. 

\subsection{Geometric fixed points}\label{SS:GFP}

As outlined in the previous section, the Bredon homology groups of a $G$-space $X$ are the (categorical) fixed points of the $G$-spectrum $\sus^\infty_G X\sma H\uM$. The \emph{geometric fixed points} construction is related to categorical fixed points (in a way that will be made precise in \cref{SS:ISS}), but has properties that make it more amenable to computation. We refer the reader to \cite[Sec. 2.5]{HHR16} for proofs of the results in this section. 

\begin{definition} \label{Def:PhiG}
Let $G$ be a finite group and let $X$ be a $G$-spectrum. Let $\mathcal{P}$ be the family of proper subgroups of $G$, and $E\mathcal{P}$ be the unique $G$-homotopy type such that
$$ E\mathcal{P}^H \simeq \begin{cases} \emptyset & H=G, \\ * & H<G. \end{cases} $$
Consider the collapse map $E\mathcal P_+\to S^0$ and denote by $\widetilde{E\mathcal P}$ its cofiber. The \emph{geometric fixed points} of $X$ are given by
$$\Phi^GX = (\widetilde{E\mathcal{P}} \wedge X)^G. $$
\end{definition}

The key properties of the geometric fixed points functor are that it commutes with taking suspension spectra and is strong symmetric monoidal. These properties are the content of the following propositions. 

\begin{proposition}
For any $G$-space $X$, there is an equivalence of spectra
$$\Phi^G \Sigma^\infty_G X \simeq \Sigma^\infty(X^G).$$
\end{proposition}

\begin{proposition}
For any two $G$-spectra $X$ and $Y$, there is an equivalence of spectra
$$\Phi^G( X \wedge Y) \simeq \Phi^G X \wedge \Phi^G Y.$$
\end{proposition}

\subsection{Proper homotopy orbits}\label{SS:PHO}

In order to relate geometric fixed points to categorical fixed points, we will need the following auxiliary construction.

\begin{definition}
The \emph{proper homotopy orbits} of a $G$-spectrum $X$ are given by
$$X_{h\mathcal{P}} = (E\mathcal{P}_+ \wedge X)^G.$$
\end{definition}







\subsection{Isotropy separation sequence}\label{SS:ISS}

In this section, we prove \cref{Lemma:Bredon-to-PhiG}, a central tool in our analysis of Bredon homology. It is based on the following cofiber sequence, which allows us to understand categorical fixed points, provided we can understand geometric fixed points and proper homotopy orbits.

\begin{lemma}[Isotropy Separation Sequence {\cite[Section 2.5.2]{HHR16}}]
The collapse map of \cref{Def:PhiG} induces a cofiber sequence of spectra
$$ X_{h\mathcal P} \to X^G \to \Phi^G X. $$
\end{lemma}

\begin{lemma}\label{Lemma:Bredon-to-PhiG}
Let $G$ be a finite group, let $d\geq 0$, and let $X$ be a $G$-spectrum. Let
$\mathcal{C}$ be a Serre class. Assume $\pi_q \Phi^K(H\underline{\Z}\sm X)\in
\mathcal{C}$ for all $K\leq G$, $q\leq d+1$. Then $H_q^K(X;\underline{\Z})\in
\mathcal{C}$ for all $K\leq G$, $q\leq d$.
\end{lemma}

The proof relies on the following homological fact, the proof of which we learned from Piotr Pstr{\k{a}}gowski. 

\begin{lemma}[Four Lemma mod Serre]\label{Lem:FourModC}
Suppose we have a commutative diagram
\[
\begin{tikzcd}
A \arrow{r}{f} \arrow{d}{\ell} & B \arrow{r}{g} \arrow{d}{m} & C \arrow{r}{h} \arrow{d}{n} & D \arrow{d}{p} \\
A' \arrow{r}{r} & B' \arrow{r}{s} & C' \arrow{r}{t} & D'
\end{tikzcd}
\]
in an abelian category $\mathcal{A}$. Let $\mathcal{C} \subseteq \mathcal{A}$ be a Serre subcategory. Suppose the rows are exact, $\ell$ is an epimorphism mod $\mathcal{C}$ and $p$ is a monomorphism mod $\mathcal{C}$. If $n$ is an epimorphism mod $\mathcal{C}$ then so is $m$. If $m$ is a monomorphism mod $\mathcal{C}$ then so is $n$.
\end{lemma}

\begin{proof}
We just prove the first statement, as the second is analogous.
As $\mathcal{C}$ is a Serre subcategory, there exists an abelian category $\mathcal{A}/\mathcal{C}$ and an exact functor $F: \mathcal{A} \to \mathcal{A}/\mathcal{C}$ which is essentially surjective with kernel $\mathcal{C}$ \cite[Lemma 02MS]{Stacks}. Applying $F$ yields a commutative diagram
\[
\begin{tikzcd}
FA \arrow{r}{Ff} \arrow{d}{F\ell} & FB \arrow{r}{Fg} \arrow{d}{Fm} & FC \arrow{r}{Fh} \arrow{d}{Fn} & FD \arrow{d}{Fp}\\
FA' \arrow{r}{Fr} & FB' \arrow{r}{Fs} & FC' \arrow{r}{Ft} & FD' 
\end{tikzcd}
\]
in the abelian category $\mathcal{C}/\mathcal{A}$ in which the rows are exact, $F\ell$ and $Fn$ are epimorphisms, and $Fp$ is a monomorphism. The classical Four Lemma implies that $Fm$ is an epimorphism, so $F\coker(m)=\coker(Fm)=0$, i.e., $\coker(m)$ lies in $\mathcal{C}$, i.e., $m$ is an epimorphism mod $\mathcal{C}$. 
\end{proof}

\begin{lemma} \label{lem:sseq-iso-mod-serre}
Let $\mathcal{C}\subseteq \mathcal{A}$ be a Serre category.
Let $\sigma:E'_r \to E''_r$ be a map of spectral sequences that is an isomorphism modulo $\mathcal{C}$ on the $E_r$ page for some $r$. Then $E'_{r+1} \too{\sigma} E''_{r+1}$ is an isomorphism mod $\mathcal{C}$.
\end{lemma}

\begin{proof}
This is a straightforward consequence of Lemma \ref{Lem:FourModC}. First apply Lemma \ref{Lem:FourModC} to the diagram
$$ \xymatrix{
0\ar[r]\ar[d] & Z'_r\ar[r]\ar[d] & E'_r\ar[r]^-{d_r}\ar[d]^\sigma & E'_r\ar[d]^\sigma
\\0\ar[r] & Z''_r\ar[r] & E''_r\ar[r]^-{d_r} & E''_r
}$$
(where $Z'_r$ denotes the submodule of cycles) to show that $Z'_r\to Z''_r$ is an isomorphism mod $\mathcal{C}$. Similar arguments involving the exact sequences $0\to Z'_r \to E'_r \to B'_r \to 0$ and $0\to B'_r \to Z'_r \to Z'_r/B'_r= E'_{r+1}\to 0$ show that the induced maps $Z'_r\to Z''_r$ and $E'_{r+1}\to E''_{r+1}$, respectively, are isomorphisms mod $\mathcal{C}$.
\end{proof}

\begin{proof}[Proof of \cref{Lemma:Bredon-to-PhiG}]
By induction up the subgroup lattice. For the base case, 
$$ H^e_q(X; \underline{\Z}) = \pi_q((H\underline{\Z}\sm X)^e) =
\pi_q(\Phi^e(H\underline{\Z}\sm X)) \in \mathcal{C}. $$
Now let $\{ e \} < K\leq G$ and let $\mathcal{P}$ denote the family of proper
subgroups of $K$.
The isotropy separation sequence
$$ (H\underline{\Z}\sm X)_{h\mathcal{P}}\to
(H\underline{\Z}\sm X)^K\to \Phi^K(H\underline{\Z}\sm X)$$
gives rise to a long exact sequence
\begin{gather*}
\cdots \to \pi_{q+1}(\Phi^K(H\underline{\Z}\sm X))\to \pi_q((H\underline{\Z}\sm X)_{h\mathcal{P}}) \to H^K_q(X; \underline{\Z}) 
\\\to \pi_q(\Phi^K(H\underline{\Z}\sm X)) \to \pi_{q-1}((H\underline{\Z}\sm X)_{h\mathcal{P}}) \to \cdots. 
\end{gather*}
To show that the middle term is in $\mathcal{C}$, we claim it suffices to show that all
the other terms are in $\mathcal{C}$: apply Lemma \ref{Lem:FourModC}
to the map from the above exact sequence to the zero exact sequence.

The assumption says that the first and fourth terms are in $\mathcal{C}$, so we
focus on showing the second term is in $\mathcal{C}$, which would also show the fifth
term is in $\mathcal{C}$ because $q\leq d$ is arbitrary.
There is a cell decomposition of $E\mathcal{P}_+$ by a finite collection of
cells
$$ \{ {K/H_i}_+ \sm S^{n_i} \}_{i\in \mathcal{I}} $$
with every $H_i < K$ (cf. \cite[Section 2.5.2]{HHR16}). Filtering $E\mathcal{P}_+\sm X$ by the cells of
$E\mathcal{P}_+$ gives rise to a spectral sequence
$$ E_1^{i,p-n_i} \cong H_p^K({K/H_i}_+\sm S^{n_i}\sm X)
\Rightarrow H^K_{p-n_i}(E\mathcal{P}_+\sm X). $$
We have
\begin{align*}
H^K_p({K/H}_+\sm S^n\sm X)  & \cong [S, H\underline{\Z}\sm K/H\sm X]^K_{p-n}
\cong [{K/H}_+, H\underline{\Z}\sm X]^K_{p-n}
\\ & \cong [S, H\underline{\Z}\sm X]^H_{p-n} \cong H^H_{p-n}(X;\underline{\Z})
\end{align*}
using self-duality of $K/H_+$. In particular, since $H_i<K$, 
we have $E_1^{i,q}\in \mathcal{C}$ by the induction hypothesis. By Lemma
\ref{lem:sseq-iso-mod-serre}, $H^K_q(E\mathcal{P}_+\sm X)$ is also in
$\mathcal{C}$.
\end{proof}

\subsection{Geometric fixed points of Eilenberg--MacLane $G$-spectra}\label{SS:GEM}

In light of the previous section, one piece of understanding the Bredon homology $(H\uM\sma \sus^\infty_GX)^G$ of a $G$-space $X$ is understanding a geometric fixed points term of the form
$$ \Phi^G(H\uM\sma \sus^\infty_GX) \simeq \Phi^G H\uM \sma \Phi^G \sus^\infty_GX \simeq \Phi^G H\uM \sma \sus^\infty X^G, $$ 
which is the $\Phi^G H\uM$-homology of the space $X^G$. In many cases of interest (see Proposition \ref{Prop:Green}), this splits into (nonequivariant) Eilenberg-Maclane spaces, allowing us to reduce to studying ordinary homology of fixed point spaces.

Recall that the category of Mackey functors is symmetric monoidal with unit the Burnside Mackey functor $\uA$ and monoidal product the \emph{box product}. A Green functor is a monoid in the category of Mackey functors \cite{lewis-green}.

\begin{proposition}\label{Prop:Green}
Let $\uM$ be a Green functor. For all $H \leq G$, the spectrum $\Phi^H H\uM$ is a connective generalized Eilenberg--MacLane spectrum, i.e., it splits as a wedge of nonnegative suspensions of Eilenberg--MacLane spectra. Moreover, if $\uM(G/K)$ is finitely generated as an abelian group for all $K \leq H$, then $\pi_k \Phi^H H\uM$ is a finitely generated abelian group for each $k \in \z$. 
\end{proposition}

\begin{proof}
The lax monoidal transformation $(-)^H \to \Phi^H(-)$ induces a map of commutative ring spectra (cf. \cite[\S B.10.5]{HHR16})
$$(H\uM)^H \to \Phi^H H\uM.$$
Since $H\uM$ is Eilenberg--MacLane, $(H\uM)^H \simeq H\uM(G/H)$, and since $\uM$ is a Green functor, $\uM(G/H)$ is a commutative ring (see, e.g., \cite[\S 2.2.5]{Shu10}). Therefore we have a map of commutative ring spectra
$$H\z \to H\uM(G/H) \simeq (H\uM)^H \to \Phi^H H\uM$$
which determines an $H\z$-module structure on $\Phi^H H\uM$. The claim that $\Phi^H H\uM$ is generalized Eilenberg--MacLane then follows from the classical fact (usually attributed to Adams) that any $H\z$-module is generalized Eilenberg--MacLane. 

That $\Phi^H H\uM$ is connective follows from the isotropy separation sequence: we have
$$\Phi^H H\uM \simeq \cofib((H\uM)_{h\mathcal{P}} \to (H\uM)^{H})$$
and both $(H\uM)_{h\mathcal{P}}$ and $(H\uM)^H$ are connective, so $\Phi^H H\uM$ is connective. 

Finite generation also follows from the isotropy separation sequence: the finite generation hypothesis on $\uM(G/K)$, along with the fact that $E\mathcal{P}$ admits an $H$-CW structure with finitely cells in each dimension, implies that each homotopy group of $(H\uM)_{h \mathcal{P}}$ and $(H\uM)^H \simeq H(\uM(G/H))$ is finitely generated, so the same is true for $\Phi^H H\uM$. 
\end{proof}

\begin{example}
The Burnside Mackey functor $\uA$ is the Mackey functor given by 
$$ \uA(G/K)=A(K), $$
where $A(K)$ denotes the Grothendieck ring of finite $K$-sets. The restrictions and transfers in $\uA$ are induced by restriction and induction of $K$-sets, respectively. 

The constant Mackey functor $\mfz$ has $\mfz(G/K)=\Z$ in every level. The restrictions in $\mfz$ are the identity map $\Z\to\Z$ and the transfer $\mfz(G/K)\to\mfz(G/L)$ is given by multiplication by the index $[L:K]$.

Both $\uA$ and $\mfz$ are Green functors, so $\Phi^H H\uA$ and $\Phi^H H\mfz$ are both connective generalized Eilenberg--MacLane spectra. Moreover, $\uA(G/K) = A(K)$ and $\mfz(G/K) = \z$ are finitely generated abelian groups for all $K \leq H$, so $\pi_k \Phi^H H\uA$ and $\pi_k \Phi^H H\mfz$ are both finitely generated abelian groups for all $k\geq 0$. 
\end{example}

\begin{remark}
It is possible to describe $\Phi^G H\uM$ explicitly in many cases of interest using previous computations of $\pi_*\Phi^G H\uM$ (e.g., \cite{HHR16, Kri20}), but this will not be necessary for our applications. 
\end{remark}




We conclude by recording the following fact which can be used to reduce to $p$-Sylow subgroups for coefficients in a constant Mackey functor. We will not need it later, but it can be used to simplify some arguments in the sequel if one works with constant Mackey functor coefficients. 

\begin{proposition}\label{Prop:GFPVanish}
Let $\uR$ be a constant Green functor and let $\uM$ be an $\uR$-algebra. If $G$ is not a $p$-group, then
$$\Phi^G H\uM \simeq *.$$
\end{proposition}

\begin{proof}
The $G$-spectrum $H\uM$ is an $H\uR$-algebra, and since geometric fixed points are strong symmetric monoidal, the spectrum $\Phi^G H\uM$ is a $\Phi^G H\uR$-algebra. But \cite[Prop. 11]{Kri20} implies that when $G$ is not a $p$-group, $\Phi^G H\uR \simeq *$, from which we conclude $\Phi^G H\uM \simeq *$. 
\end{proof}

\begin{example}
The constant Green functor $\mfz$ is an algebra over itself, so $\Phi^G H\mfz \simeq *$ if $G$ is not a $p$-group. On the other hand, one can show that $\pi_0 \Phi^G H\uA \cong \z$ for any finite group $G$, so $\Phi^G H\uA \not\simeq *$ for any $G$. 
\end{example}

\begin{remark}
For simplicity, we will work with $\mfz$ coefficients in the sequel. If $\uM$ is a constant Mackey functor which is levelwise finitely generated, then it is straightforward to deduce $\uM$ coefficient analogues of all of the $\mfz$ coefficient statements below using the presentation of $\uM$ as a $\mfz$-module (cf. deducing homology isomorphisms with $M$ coefficients from homology isomorphisms with $\z$ coefficients). 

For certain groups, e.g., any cyclic $p$-group $C_{p^n}$, it is also straightforward to pass from $\mfz$ coefficients to $\uA$ coefficients using induction up the subgroup lattice, the short exact sequence of Mackey functors
$$0 \to \uI \to \uA \to \mfz \to 0,$$
and the fact that $H_*^{C_{p^n}}(X;\uI) \cong H_*^{C_{p^{n-1}}}(X^{C_p};\uI')$, where $\uI'$ is the $C_{p^{n-1}}$ Mackey functor with $\uI'(C_{p^{n-1}}/C_{p^m}) = \uI(C_{p^n}/C_{p^{m+1}})$. Since every Mackey functor is an $\uA$-module, one can deduce analogues of our results with coefficients in any Mackey functor which is finitely generated for cyclic $p$-groups. 

We leave the extension to other coefficients and more complicated groups to the interested reader. 
\end{remark}

\section{Configuration spaces and stabilization maps}\label{Sec:Conf}

As in the previous section, $G$ is any finite group. In this section, we discuss configuration spaces and stabilization maps. We recall the nonequivariant theory in \cref{SS:Conf} and discuss the equivariant analogues in \cref{SS:CGS} and \cref{SS:GConf}. In \cref{SS:CGS}, we define the more nuanced notion of $S$-configuration spaces (\cref{Def:ConfG} and \cref{Def:CGS}) and their stabilization maps. Our main results are \cref{Prop:FixptConfDecomp}, which expresses the $G$-fixed points of configuration spaces in terms of $S$-configurations, and \cref{Cor:CGSDecomp} which expresses each $S$-configuration space in terms of products of ordinary configuration spaces. Combining these results yields \cref{MT:CtoCSG}.

\subsection{Nonequivariant configuration spaces and stabilization maps}\label{SS:Conf}

\begin{definition}\label{Def:Conf}
Let $M$ be a space.
The \emph{configuration space of $n$ ordered points in $M$} is
$$ \Conf_n(M):=\Emb(\{1,\ldots,n\},M)= \{ (x_1,\dots,x_n)\in M^{\x n} \st x_i \neq x_j \text{ for all
}i\neq j \} $$
topologized as a subspace of $M^{\x n}$.

The \emph{configuration space of $n$ unordered points in $M$} is the quotient space
$$C_n(M) := \Conf_n(M)/\Sigma_n,$$
where the symmetric group $\Sigma_n$ acts by permuting points. 
\end{definition}


When $M$ is an open, connected manifold with $\dim M\geq 2$, then $M$ is the interior of a manifold with boundary, and there there are well-defined \emph{stabilization maps} $C_n(M) \to C_{n+1}(M)$ (cf. \cite[Secs. 2.1-2.2]{Pal18}). Since we will describe the equivariant analogue of these maps in some detail in \cref{SS:GConf}, we review the non-equivariant construction.

\begin{definition}
We say that an $n$-dimensional manifold $M$ is \emph{stabilizable} if $M$ is homeomorphic to the interior of a $n$-dimensional manifold with nonempty boundary. 
\end{definition}

\begin{construction}
\label{Constr:nonequivariant-stabilization}
When $M$ is stabilizable, we can define a stabilization map on unordered configurations as follows. Let $M$ be stabilizable with $M \cong W^\circ$, where $W$ is an $n$-dimensional manifold with nonempty boundary. Let $p \in \partial W$. By the definition of manifold with boundary, 
there exists an open neighborhood $U \subset W$ containing $p$ together with a diffeomorphism 
$$\phi: U \xrightarrow{\cong} \r^n_+ = \{ x = (x_1,\ldots,x_n) \in \r^n : x_1 \geq 0\}$$ 
sending $p$ to $0$. Let $b: \r^n_+ \to \r^n_+$ be a smooth bump function which  sends $0$ to $(1,0,\ldots,0)$. Let $e: W \to W$ be the self-embedding defined by
$$e(x) = \begin{cases}
	x \quad & \text{ if } x \notin U, \\
	(\phi^{-1} \circ b \circ \phi)(x) \quad & \text{ if } x \in U.
\end{cases}$$
Then by construction, $e(M) \subseteq M$ and $e(p) \notin e(M)$, so we may define 
\begin{equation}\label{eq:sigma}
\sigma: C_n(M) \to C_{n+1}(M)
\end{equation}
by
$$\sigma(m_1,\ldots,m_n) = (e(m_1),\ldots,e(m_n),e(p)).$$
\end{construction}

\subsection{Equivariant configuration spaces}\label{SS:CGS}

In this section, we study equivariant configuration spaces. Our main result (\cref{Cor:CGSDecomp}) will allow us to relate these spaces back to nonequivariant configuration spaces. 

\begin{definition}\label{Def:ConfG-equivariant}
Let $X$ be a $G$-space and let $n\geq 1$. The configuration space of $n$ ordered
points in $X$ is the $G$-space of nonequivariant embeddings
$$ \Conf_n(X) = \Emb(\{ 1,\dots,n \},X). $$
The configuration space of $n$ unordered points in $X$ is the quotient $G$-space
$$ C_n(X) = \Conf_n(X)/\Sigma_n. $$
Let $C(X) = \bigsqcup_{n\geq 1} C_n(X)$ equipped with the disjoint union topology. In $\Conf_n(X)$ and $C_n(X)$, the $G$-action is induced by the action on $X$.
\end{definition}

These are the same definitions as in the nonequivariant case, but now they are
$G$-spaces. In Proposition \ref{Prop:FixptConfDecomp}, we will show that their
fixed points $C_n(X)^G$ split into pieces $C_S^G(X)$ which are described below.

\begin{definition}\label{Def:ConfG}
Let $X$ be a $G$-space and let $S$ be a finite $G$-set of cardinality $n$. The \emph{space of ordered $S$-configurations in $X$}, denoted $\Conf^G_S(X)$, is the space of $G$-equivariant embeddings $S\to X$,
$$\Conf_S^G(X) := \Emb^G(S,X).$$
\end{definition}

This is a $G$-equivariant analogue of the ordinary notion of an ordered configuration space. Indeed, if $G=e$ is the trivial group, we have
$$ \Conf^e_S(X) \cong \Conf_{n}(X^e). $$
Precomposing a $G$-equivariant embedding $S\to X$ by an automorphism of $S$ yields another such embedding. In this way, $\Conf^G_S(X)$ admits an action by the automorphism group $\Aut(S)$. 

\begin{definition} \label{Def:CGS}
The \emph{space of unordered $S$-configurations in $X$} is the quotient
$$ C^G_S(X) := \Conf^G_S(X)/\Aut(S) $$
of $\Conf^G_S(X)$ by the $\Aut(S)$-action.
\end{definition} 

Similarly, this is a $G$-equivariant analogue of the ordinary notion of unordered configuration space. If $G=e$ is the trivial group, then $\Aut(S)\cong\Sigma_{n}$, and we have
$$ C^e_S(X) = \Conf^e_S(X)/\Aut(S) \cong \Conf_{n}(X^e)/\Sigma_{n} = C_{n}(X^e).$$
We can express $C^G_S(X)$ in terms of
nonequivariant configuration spaces of a related space.

\begin{definition}[{\cite[Section I.5]{tom-dieck-transformation-groups}}] \label{Def:X_(H)}
For a $G$-space $X$ and subgroup $H\leq G$, we denote by $X_{(H)}$ the subspace 
$$ X_{(H)} = \{ x\in X : (G_x) = (H) \} \subseteq X $$
consisting of points with stabilizer conjugate to $H$.
\end{definition}

Since the stabilizers of points in a fixed orbit are conjugate, the $G$-action on $X$ restricts to a $G$-action on $X_{(H)}$. The preimage of any point under the projection $X_{(H)}\to X_{(H)}/G$ is a full orbit isomorphic to $G/H$.

\begin{proposition}\label{Prop:CGOrbits}
Given a subgroup $H\subseteq G$ and $k\geq 0$,
the space of unordered $k[G/H]$-configurations in $X$ satisfies 
$$ C^G_{k[G/H]}(X) \cong C_k \left( X_{(H)}/G \right). $$
\end{proposition}
\begin{proof}
A point in $C^G_{k[G/H]}(X)$ is a subset $S\subseteq X$ which is isomorphic as a $G$-set to $k[G/H]$. Each point in $S$ has stabilizer conjugate to $H$, so $S$ is furthermore a subset of $X_{(H)}$. The set $S$ consists of $k$ orbits of type $G/H$, so its image under the projection $X_{(H)}\to X_{(H)}/G$ is a size $k$ subset of $X_{(H)}/G$. This defines a map $$ C^G_{k[G/H]}(X) \to C_k \left( X_{(H)}/G \right), $$
which is a homeomorphism. Indeed, the preimage of any $k$ element subset of $X_{(H)}/G$ is a subset $S\subseteq X_{(H)}\subseteq X$ of the above form, which defines an inverse map.
\end{proof}

The space $X_{(H)}/G$ can also be expressed in terms of $H$-fixed points of $X$ and their $W_GH$-action, as we show in the following lemma.
\begin{lemma}\label{Lem:modW}
For any $G$-space $X$,
$$ X_{(H)}/G \cong \Big( X^H \setminus \bigcup_{K>H} X^K \Big)/W_GH. $$
\end{lemma}
\begin{proof}
There is a $G$-equivariant homeomorphism
\begin{equation}\label{eq:phi} \varphi: X_{(H)} \to G \times_{N_GH} \Big( X^H \setminus
\bigcup_{K>H} X^K \Big) \end{equation}
and hence on orbits we have
\begin{align*}
X_{(H)}/G & 
\cong \Big( X^H \setminus \bigcup_{K>H} X^K \big)/N_GH
\cong \Big( X^H \setminus \bigcup_{K>H} X^K \big)/W_GH
\end{align*}
since $H\leq N_GH$ acts trivially on $X^H$.
\end{proof}

In Corollary \ref{Cor:CGSDecomp} we generalize Proposition \ref{Prop:CGOrbits} to a formula for $C^G_S(X)$ for an arbitrary $G$-set $S$. It follows from the next lemma, whose proof is omitted.

\begin{lemma}\label{Lem:CGProduct}
Let $S$ and $T$ be finite $G$-sets such that no orbit of $S$ is isomorphic to an orbit of $T$. Then,
$$ C^G_{S+T}(X)\cong C^G_S(X)\times C^G_T(X). $$
\end{lemma}

Combining \cref{Prop:CGOrbits} and \cref{Lem:CGProduct}, we obtain our desired decomposition formula.

\begin{corollary} \label{Cor:CGSDecomp}
Let $S$ be the finite $G$-set
$$ S = \coprod_{(H)} k_{(H)}[G/H], $$
where $k_{(H)}\geq 0$ for each conjugacy class of subgroups $(H)$. Then,
$$ C^G_S(X) \cong \prod_{(H)} C_{k_{(H)}} (X_{(H)}/G ). $$
\end{corollary}



Recall that $C_n(X)$ denotes the $G$-space of nonequivariant embeddings of $n$
points. The next proposition is the main result of this section, which describes
its fixed points in terms of the spaces $C_S^G(X)$ we have been studying above,
in the case that $X$ is a manifold.

\begin{proposition}\label{Prop:FixptConfDecomp}
For a $G$-manifold $M$ and $n\geq 0$, there is a disjoint union decomposition
$$ C_n(M)^G \cong \coprod_{|S|=n} C^G_S(M), $$
where the disjoint union runs over isomorphism classes of finite $G$-sets $S$ with cardinality $n$.
\end{proposition}

\begin{proof}
Let $S = \coprod_{i\in I} G/H_i$ be a $G$-set of cardinality $n$.
Since $C_S^G(M)$ and $C_n(M)^G$ are given compatible quotient topologies, it
suffices to show that $\Conf_S^G(M)$ is open in $\Conf_n(M)^G$.
Let $x\in \Conf_S^G(M)$ be an ordered configuration consisting of orbits
$\mathcal{O}_i= \{ g\cdot x_i \}_{g\in G/H_i}$ for $i\in I$.
By the equivariant slice theorem (see e.g. \cite[Theorem I.2.1.1]{audin}),
there are open neighborhoods $U_{i,g}\ni g\cdot x_i$
and diffeomorphisms $\phi_i:\bigsqcup_g U_{i,g} \cong G\x_{H_i} T_{x_i}M$ with
$\phi_i(g\cdot x_i)= (g, 0)$. Then $U = \prod_{i,g}
U_{i,g}\subseteq M^{\x n}$ is an open neighborhood of $x$.

An element of $\Conf_n^G(M)\ints U$ is an $n$-tuple $(y_{i,g})$ where
$y_{i,g}\in U_{i,g}$, and
we may assume that the $U_{i,g}$'s are disjoint in $M$.
It suffices to show that $y_{i,g}$ has isotropy $H_i$.
Since each $U_{i,g}$ contains a single $y_{i,g}$ and $G$ acts on $\{ y_{i,g} \}$,
an element $g'\in G$ fixes $y_{i,g}$ if and only if it preserves $U_{i,g}\cong
\{ g \}\x T_{x_i}M \subseteq G\x_{H_i} T_{x_i}M$, and this occurs if and only if $g'\in H_i$.
\end{proof}

\begin{remark}
Though Proposition \ref{Prop:FixptConfDecomp} uses the theory of
(finite-dimensional) manifolds, we will investigate an analogous statement for $\rho^\infty$ in future work
and identify the pieces:
\begin{equation}\label{eq:rho-infty} C^G_{G/H}(\rho^\infty)\hteq B(W_GH). \end{equation}
These spaces are interesting because
Guillou and May \cite[Lem. 1.2]{GM17} identify $C_n(\rho^\infty)$ with the $n$-th
space in the $G$-equivariant little $\rho^\infty$-disks operad, which is
well-known to be $G$-equivalent to $B_G\Sigma_n$ (cf. \cite[Pg. 489]{CW91}).
In particular, the splitting here recovers the Lashof--May splitting
\cite[Theorem VII.2.7]{May96} of $(B_G\Sigma_n)^G$.
\end{remark}

\subsection{Equivariant stabilization maps}\label{SS:GConf}
Our goal in this subsection is to define an equivariant version of Construction
\ref{Constr:nonequivariant-stabilization}.

\begin{definition}\label{Def:HStab}
Let $H \leq G$ be a subgroup. A $G$-manifold $M$ is \emph{$H$-stabilizable} if
it is equivariantly homeomorphic to the interior of a $G$-manifold $W$ with boundary containing a point of isotropy group $H$. 
\end{definition}

\begin{construction}\label{Constr:sigmaGH}
Suppose $M$ is an $H$-stabilizable $G$-manifold of dimension $n$, and write $M \isom W^\circ$ where $W$ is a $G$-manifold with boundary containing a point $p \in \partial W$ with isotropy group $H$. By the equivariant slice theorem for manifolds with boundary \cite[Theorem 3.6]{Kan07}, there exists an open subset $U \subseteq W$ containing the orbit $\{ g\cdot p \st g\in G/H \}$ together with a $G$-equivariant diffeomorphism 
$$\phi: U\to \r_{\geq 0} \times (G\x_H V) $$ 
sending $g\cdot p$ to $(0, g,0)$, where $\r_{\geq 0}$ has the trivial action and $V \cong T_p(\partial M)$ is the $G$-representation determined by the action of $G$ on the tangent space of $\partial M$ at $p$. Let $b: \r_{\geq 0} \times (G\x_HV) \to \r_{\geq 0} \times (G\x_HV)$ be a smooth, $G$-equivariant bump function which is the identity in the $G\x_H V$ component and sends $(0,g,0)$ to $(1,g,0)$. Let $e: W \to W$ be the $G$-equivariant self-embedding defined by
$$e(x) = \begin{cases}
	x \quad & \text{ if } x \notin U, \\
	(\phi^{-1} \circ b \circ \phi)(x) \quad & \text{ if } x \in U.
\end{cases}$$
Let $g_1,\ldots,g_r$ range over the elements of $G/H$. We define 
\begin{equation}\label{eq:sigmaGH} \sigma_{G/H}: C_n(M) \to C_{n+|G/H|}(M)\end{equation}
by
$$\sigma(m_1,\ldots,m_n) = (e(m_1),\ldots,e(m_n),e(p),g_1e(p),\ldots,g_re(p)).$$
\end{construction}

\begin{remark}
Observe that in \cref{Constr:sigmaGH}, the bump function $\phi$, and hence $\sigma_{G/H}$, preserves isotropy type. Moreover, since $\sigma_{G/H}$ adds a single $G/H$ orbit, it descends to maps
\begin{align*}
C_n(M/G) & \to C_{n+1}(M/G)
\\C_n(M)^G & \to C_{n+|G/H|}(M)^G
\\C_S^G(M) & \to C_{S+[G/H]}^G(M).
\end{align*}
\end{remark}

Putting together Corollary \ref{Cor:CGSDecomp} and Proposition
\ref{Prop:FixptConfDecomp}, we expect equivariant homological stability of $M$ should be
related to nonequivariant homological stability of $M_{(H)}/G$. First (Lemma
\ref{lem:M_(H)/G-open-manifold}) we check that $M_{(H)}/G$ satisfies the
hypotheses of nonequivariant homological stability except for connectedness,
which is an additional assumption and not needed for the definition of
stabilization maps. Then (Lemma
\ref{Lem:sigma-comparison}) we check that the stabilization maps agree. In the
next section (\cref{Thm:CGSHomStab}) we will put these together to prove a form
of equivariant homological stability.

\begin{lemma} \label{lem:M_(H)/G-open-manifold}
If $M$ is an $H$-stabilizable smooth $G$-manifold for $H\leq G$, then
$M_{(H)}/G$ is an open smooth manifold.
\end{lemma}
\begin{proof}
Let $W = W_HG$.
To see that $M_{(H)}/G$ is a manifold, note that $M_{(H)}/G = M_{(H)}/W$ since
$M_{(H)}$ consists entirely of $G/H$-orbit types. Thus $W = \Aut(G/H)$ acts
freely on the manifold $M_{(H)}$, and its acts properly because $W$ is finite. Thus
$M_{(H)}/W$ is a (smooth) manifold.

For noncompactness, first we show noncompactness of $M_{(H)}$. The stabilizability assumption says that there is a
neighborhood of the manifold with boundary $W$ that is diffeomorphic to
$\r_{\geq 0}\x (G\x_H V)$, where $V$ is the tangent space in $\partial W$ of a
point $p\in \partial W$ of isotropy $H$ (see Construction \ref{Constr:sigmaGH}).
The subspace of isotropy $H$ in this neighborhood is $r_{\geq 0}\x (G\x_H V^H)$,
and taking away the boundary, we see that $M_{(H)}$ is missing a limit point,
namely $\{ 0 \}\x (G\x_H \{ 0 \}) = p$. Now passing to orbits $M_{(H)}/G$, we
have a neighborhood $\r_{>0}\x V^H$ with limit point $(0,0)$ that is also not in
$M_{(H)}/G$.
\end{proof}



\begin{lemma}\label{Lem:sigma-comparison}
Suppose $M$ is $H$-stabilizable.
Then the equivariant stabilization map \eqref{eq:sigmaGH} induces a map
$$ \bar{\sigma}_{G/H}: C_n(M_{(H)}/G)\to C_{n+1}(M_{(H)}/G) $$
that agrees with the nonequivariant stabilization map \eqref{eq:sigma}. 
\end{lemma}
Note that Lemma \ref{lem:M_(H)/G-open-manifold} is used to guarantee the map
\eqref{eq:sigma} exists.
\begin{proof}
Since $\sigma_{G/H}:C_n(M)\to C_{n+|G/H|}(M)$ preserves isotropy type and adds a $G/H$ orbit, it induces a map $\bar{\sigma}_{G/H}^e:C_n(M_{(H)}/G)\to C_{n+1}(M_{(H)}/G)$, and the equivariant bump function $b$ induces a function $\bar{b}$ on $M_{(H)}/G$. First note from Lemma \ref{lem:M_(H)/G-open-manifold} that $M_{(H)}/G$ is a manifold.
To check $\bar{\sigma}_{G/H}$ agrees with the classical stabilization map, we must check that $\bar{b}$ is a bump function. 
It is clear that it has compact support. To check it is smooth, first note that the restriction of the smooth function $b$ to $M_{(H)}$ is smooth. To check it is smooth after taking orbits, recall that every orbit in $M_{(H)}$ has a neighborhood diffeomorphic to $G\x_H \r^n$, and taking orbits simply collapses the $|G/H|$ Euclidean components to one Euclidean component.
\end{proof}

\section{Equivariant homological stability}\label{Sec:Stab}

In this section, we prove our main theorems, \cref{MT:BredonFails} and \cref{MT:Stability}. Our proof is by reduction to the classical stability result of McDuff--Segal (\cref{Thm:HomStab}), which we recall in \cref{SS:MS}. We also discuss a weaker notion of homological stability (\cref{Cor:WeakStability}) in terms of finite generation; the relationship between finite generation stability and sequential stability is made precise in \cref{Prop:AlgStability}. Using nonequivariant homological stability, we prove \cref{MT:CSGStability} (\cref{Thm:CGSHomStab}) that the equivariant stabilization maps eventually induce isomorphisms in the homology of $S$-configuration spaces in \cref{SS:CGSStability}. We then use our decomposition of the $G$-fixed points of configuration spaces into $S$-configuration spaces, along with our Bredon-to-geometric-fixed-point reduction result (\cref{Lemma:Bredon-to-PhiG}), to prove Bredon homological stability in \cref{SS:BredonStab}.

\subsection{Nonequivariant homological stability}\label{SS:MS}

\begin{theorem}[{McDuff--Segal, as stated in \cite[Thm. 1.2]{Pal18} with $X=*$}]\label{Thm:HomStab}
Assume $M$ is an open, connected manifold with $\dim(M) \geq 2$. The stabilization map $\sigma:C_n(M) \to C_{n+1}(M)$ induces an isomorphism on integral homology in degrees $* \leq \frac{n}{2}.$
\end{theorem}

Our goal for the rest of this subsection is to relate homological stability with respect to iterated applications of $\sigma$ with finite generation over $\z[\sigma_*]$. Since this does not depend on the fact that the groups that stabilize are homology groups, we take an abstract approach, presenting the translation for arbitrary sequences of modules that stabilize.

Fix a commutative noetherian ring $R$ and let 
$$ A_0 \xrightarrow\sigma A_1 \xrightarrow\sigma A_2 \xrightarrow\sigma \cdots $$
be a sequence of $R$-modules $A_n$ for $n\geq 0$ connected by maps $\sigma_n:A_n\to A_{n+1}$. Write 
$$ A = \bigoplus_{n\geq 0} A_n. $$
The collection of maps $\sigma_n$ assemble into an endomorphism $\sigma:A\to A$, and this map endows $A$ with the structure of a $R[\sigma]$-module. The next proposition relates finiteness conditions on $A$ and $\sigma$ to a stability condition on the sequence.

\begin{proposition} \label{Prop:AlgStability}
With $R$, $A$ and $\sigma$ as above, the following are equivalent:
\begin{enumerate}
\item $A_n$ is a finitely generated $R$-module for each $n$, and the sequence stabilizes: that is, there exists $N\geq 0$ such that for all $n\geq N$, $\sigma:A_n\to A_{n+1}$ is an isomorphism; and
\item $A$ is finitely generated as an $R[\sigma]$-module.
\end{enumerate}
\end{proposition}

\begin{proof}
To see (1) $\Rightarrow$ (2), let $f_n:R^{r_n}\to A_n$ be a surjection. Define a map
$$ f:R[\sigma]^{r_0} \oplus \cdots \oplus R[\sigma]^{r_N} \to A $$
which maps the generators of $R[\sigma]^{r_n}$ to $A_n\subseteq A$ according to $f_n$ and extends $R[\sigma]$-linearly. To see that $f$ is surjective, note that by construction it surjects onto $\bigoplus_{n=0}^N A_n$, and it surjects onto $\bigoplus_{n=N}^\infty A_n \cong A_N \otimes \Z[\sigma]$ because it is a $\sigma$-linear map that surjects onto $A_N$.

To see (2) $\Rightarrow$ (1), let $f:R[\sigma]^r\to A$ be a surjection from a free $R[\sigma]$-module of finite rank $r$. Since $f(\sigma R[\sigma]^r)\subseteq \sigma A$, $f$ descends to a surjective map on quotients 
$$ \bar{f}: R^r \cong R[\sigma]^r/\sigma R[\sigma]^r \to A/\sigma A = \coker\sigma$$
which exhibits $\coker\sigma$ as a finitely generated $R$-module. 
Note that 
$$ \coker\sigma \cong \bigoplus_{n\geq 0} (\coker\sigma_n). $$
That $\coker\sigma$ is finitely generated thus implies that $(\coker\sigma)_n=0$ for sufficiently large $n$. It follows that the sequence is eventually surjective, that is, there exists a $M$ such that $\sigma_n$ is surjective for $n\geq M$. If we let
$$ K_i = \ker(A_M \to A_i), $$
for $i\geq M$, we obtain an increasing chain $K_M\subseteq K_{M+1}\subseteq\cdots$ of submodules of $A_M$. As a finitely generated module over the noetherian ring $R$, $A_M$ is noetherian. Therefore, there must exist an $N\geq 0$ at which point this increasing chain stabilizes, that is $K_N=K_{N+1}=\cdots$. It follows that for $n\geq N$, $\sigma_n$ is an isomorphism.

Now we show that $A_n$ is finitely generated over $R$ by induction on $n$. Since
$\coker(\sigma)|_{A_0} = A_0$, we must have that $A_0$ is finitely generated.
Now assume $n\geq 1$ and assume inductively that $A_{n-1}$ is finitely
generated. Consider the exact sequence
$$ A_{n-1}\too{\sigma} A_n \too{\sigma} \coker(\sigma)|_{A_n}\to 0. $$
Since $A_{n-1}$ and $\coker(\sigma)|_{A_n}$ are finitely generated by
assumption, we have that $A_n$ is also finitely generated.
\end{proof}

\begin{corollary}\label{Cor:WeakStability}
Let $M$ be a stabilizable open, connected manifold with $\dim(M) \geq 2$. If $H_d(C_k(M))$ is finitely generated as an abelian group for all $d \geq 0$, then
$H_*(\bigsqcup_{k \geq 0} C_k(M))$ is finitely generated as a $\z[\sigma_*]$-module. 
\end{corollary}

We note that the finite generation of $H_d(C_k(M))$ in \cref{Cor:WeakStability} follows from the finite generation of $H_d(M)$. First we need a lemma.

\begin{lemma} \label{lem:homology-orbits-finite}
Suppose $X$ has a free $G$-action, and $H_d(X;\Z)$ is a finitely generated
abelian group for all $d$. Then $H_d(X/G;\Z)$ is finitely generated.
\end{lemma}
\begin{proof}
There is a fibration $X\to EG\x_G X \to BG$, and
since the action of $G$ on $X$ is free, we have $EG\x_G X\hteq
X/G$. There is a Serre spectral sequence with local coefficients
$$ E_2^{*,*}= H_*(BG; H_*(X;\Z)) \Rightarrow H_*(X/G). $$
The $E_2$-page is the cohomology of the complex $C_*(EG;\Z)\otimes_{\Z[G]}
H_*(X;\Z)$ which is finitely generated in each degree by assumption.
Thus the $E_2$-term is finitely generated in each degree, and so is the
target.
\end{proof}

\begin{proposition}\label{Prop:MtoCnM}
If $H_d(M)$ is finitely generated for all $d \geq 0$, then so is $H_d(C_n(M))$ for all $d,n \geq 0$. 
\end{proposition}
\begin{proof}
By Lemma \ref{lem:homology-orbits-finite}, it suffices to show that
$H_d(\Conf_n(M);\Z)$ is finitely generated. Since $\Conf_n(M)$ is a manifold,
its homology is Poincar\'e dual to cohomology with compact support
$H^d_c(\Conf_n(M);\mathcal{O})$ with coefficients in the orientation sheaf $\mathcal{O}$ (see
\cite[Example 2.9]{bredon-sheaf-theory}).
To study this, we apply \cite[Theorem 5.18]{petersen-generalized}, which gives a
spectral sequence that in our case specializes to
$$ E_1^{p,q} = \bigoplus_\attop{T\in J_{U_0}\\|T| = n-p}\bigoplus_{i+j=p+q} \til{H}^i(J(T);
H^j_c(\overline{M(T)}; \mathcal{O})) \implies H_c^{p+q}(\Conf_n(M); \mathcal{O}) $$
where $J_{U_0}$ is a subset of the set of partitions of $\{ 1,\dots,n \}$,
$J(T)$ is a finite CW complex, and $\overline{M(T)}$ is a closed submanifold of 
$$ \{ (x_1,..,x_n)\in M^n \st x_i=x_j \iff \text{$x_i = x_j$ if $i$ and $j$ are in the
same block of $T$} \}, $$
which is itself a closed submanifold of $M^n$. Closed manifolds have cohomology
that is finitely generated in each degree. Since the $E_1$ page of the spectral
sequence is finitely generated for each $(p,q)$ and $p,q\geq 0$ we have that the
target is finitely generated in each degree.
\end{proof}

\subsection{Instability for adding a trivial orbit}

If $M$ is $G$-stabilizable, then we may add a fixed point to obtain a $G$-equivariant stabilization map
$$\sigma_{G/G}: C_n(M) \to C_{n+1}(M)$$
whose underlying map
$$\sigma_{G/G}^e = \sigma : C_n(M) \to C_{n+1}(M)$$
is the classical stabilization map. In this section, we show that the Bredon homology of $C(M)$ cannot be stable with respect to this stabilization map alone. For simplicity of exposition, we consider the degree 0 case where $G=C_p$.




\begin{proposition} \label{Prop:H0Unstable}
For $G=C_p$ and $M$ an open $G$- and $e$-stabilizable manifold, the $\z[(\sigma_{G/G})_*]$-module $H_0^G(C(M);\mfz)$ is not finitely generated. 
\end{proposition}
\begin{proof}
The cellular chain complex $C^G_*(C(M))$ computing Bredon homology (cf. \cref{SS:Bredon}) is given by
\begin{align*}
C^G_*(C(M)) &= \Big( \sum_{G/H} C_*(C(M)^H)\otimes\mfz(G/H) \Big)/\sim\\
&= \Big( C_*(C(M)^e) \otimes \mfz(C_p/e) \oplus C_*(C(M)^G) \otimes \mfz(C_p/C_p) \Big) /\sim,
\end{align*}
In the case $G=C_p$, the orbit category is generated by the projection $r:C_p/e\to C_p/C_p$ and automorphism $\gamma:C_p/e\to C_p/e$ which corresponds to the action of the generator of $C_p$. Therefore, in $C^G_*(C(M))$, we have the relations:
\begin{itemize}
\item $\gamma\cdot x\sim x$ for $x\in C_*(C(M)^e)$ and $\gamma\in C_p$ the generator of $C_p$, and 
\item $r^*x\sim px$ for $x\in C_*(C(M)^G)$, where $r^*:C(M)^G\to C(M)^e$ is the inclusion of fixed points.
\end{itemize} 
From \cref{Prop:FixptConfDecomp}, we have that the components of $C(M)^G$ are indexed by finite $G$-sets, so we have
$$ H_0(C(M)^{C_p}) \cong \ZZ[\sigma]\{ x_0, x_1, x_2, \ldots \}, $$
where $x_i$ represents the component corresponding to the $G$-set $i\cdot C_p/e$. The components of $C(M)$ are indexed by natural numbers, so we have 
$$ H_0(C(M)^e) \cong \ZZ[\sigma]\{ y_0 \}. $$
We conclude that $H^G_0(C(M))$ is given by
$$ H^G_0(C(M)) \cong \frac{ \ZZ[\sigma]\{ x_0,x_1,x_2,\cdots, y_0 \} }{ px_i = \sigma^i y_0}, $$
which is not finitely generated as a $\ZZ[\sigma]$-module.
\end{proof}

\subsection{Homological stability for $S$-configurations}\label{SS:CGSStability}


As above, let $M$ be a $G$-manifold. In this section, we fix a conjugacy class of subgroups $(H)$ and study stabilization maps for $[G/H]$ orbits. We will combine Corollary \ref{Cor:CGSDecomp} with classical homological stability applied to the nonequivariant space $M_{(H)}/G$, to obtain homological stability for $C_S^G(M)$.

\begin{theorem}\label{Thm:CGSHomStab}
Let $M$ be a $G$-manifold and $S$ a finite $G$-set containing $k$
copies of the orbit
$[G/H]$ for a conjugacy class of subgroups $(H)$. If $M$ is $H$-stabilizable and
$M_{(H)}/G$ is connected,
then the $(H)$-stabilization map of \eqref{eq:sigmaGH} induces an isomorphism
$$ (\sigma_{G/H})_*: H_d(C^G_S(M)) \to H_d(C^G_{S+[G/H]}(M)) $$
in integral homology in degrees $d\leq k/2$.
\end{theorem}

\begin{proof}
Decompose the finite $G$-set $S$ as 
$$ S \cong k[G/H] \sqcup S', $$
where $S'$ is a finite $G$-set with no orbits isomorphic to $[G/H]$. \cref{Cor:CGSDecomp} allows us to decompose the space of unordered $S$-configurations in $M$ as 
$$ C^G_S(M) \cong C^G_{k[G/H]}(M) \times C^G_{S'} \cong C_k(M_{(H)}/G) \times C^G_{S'}(M). $$
By \cref{Cor:CGSDecomp}, we have 
\begin{align*}
H_*(C^G_S(M)) &\cong H_*(C_k(M_{(H)}/G)\times C^G_{S'}(M)), \\
H_*(C^G_{S+[G/H]}(M)) &\cong H_*(C_{k+1}(M_{(H)}/G)\times C^G_{S'}(M)).
\end{align*}

By naturality, we obtain a map of K\"unneth short exact sequences
$$ \adjustbox{scale=0.85,center}{
\begin{tikzcd}
\displaystyle\bigoplus_{i+j=d-1} \Tor_1\left(H_i(C_k(M_{(H)}/G)), H_j(C^G_{S'}(M))\right) \arrow[r] \arrow[d]
& \displaystyle\bigoplus_{i+j=d-1} \Tor_1\left(H_i(C_{k+1}(M_{(H)}/G)), H_j(C^G_{S'}(M))\right) \arrow[d] \\
H_d(C^G_S(M)) \arrow[r, "(\sigma_{G/H})_*"] \arrow[d]
& H_d(C^G_{S+[G/H]}(M)) \arrow[d] \\
\displaystyle\bigoplus_{i+j=d} H_i(C_k(M_{(H)}/G))\otimes H_j(C^G_{S'}(M)) \arrow[r] 
& \displaystyle\bigoplus_{i+j=d} H_i(C_{k+1}(M_{(H)}/G))\otimes H_j(C^G_{S'}(M))
\end{tikzcd}
}
$$
By \cref{Lem:sigma-comparison}, the top and bottom horizontal arrows in the above diagram are given by the classical stabilization maps $\sigma:C_k(M_{(H)}/G)\to C_{k+1}(M_{(H)}/G)$
for the manifold $M_{(H)}/G$, and so \cref{Thm:HomStab} implies that these maps are isomorphisms for $d\leq k/2$.
(Note that $M_{(H)}/G$ is not empty because we have assumed that there are $k$
points of isotropy type $(H)$.)
It follows from the five lemma that $(\sigma_{G/H})_*$ is an isomorphism in the range $d\leq k/2$.
\end{proof}

\subsection{Bredon homological stability}\label{SS:BredonStab}

We now return to the situation of ordinary configuration spaces.


For $K\leq G$, let 
\[ P_K = \Z[\sigma_{K/H} : H\subseteq K] . \]
Clearly, $P_K$ acts on $K$-Bredon cohomology $H^K_*(C(M))$, but we can also
consider the action of $P_G$ on $H^K_*(C(M))$ via the restriction
map $\res^G_K: P_G\to P_K$, which is defined as follows.
Given $H,K\leq G$, the double coset formula gives an isomorphism of
$K$-sets
$$ G/H\ \ \cong \bigsqcup_{KgH\in K\backslash G/H} K/(K\ints gHg^{-1}),$$
and so
\begin{equation}\label{eq:res}\res^G_K(\sigma_{G/H})\ \  = \prod_{KgH\in K\backslash
G/H}\sigma_{K/(K\ints gHg^{-1})}.\end{equation}

\begin{lemma} \label{lem:RK-to-RG}
Suppose $G$ is a Dedekind group (i.e., every subgroup is normal).
Then for all $K\leq G$, $P_K$ is finitely generated as a $P_G$-module, where the action is via the restriction map $\res^G_K:P_G\to P_K$.
\end{lemma}
\begin{proof}
As $P_K$ is a finitely generated algebra, it suffices to show that $P_K$ is
integral over $P_G$. We will show that there is a power of every generator
$\sigma_{K/H}$ (for $K\leq H$) in the image of $\res^G_K$, which implies the same for every
monomial. Combining the Dedekind assumption with \eqref{eq:res} shows that
$\res^G_K(\sigma_{G/H}) = \sigma_{K/H}^{[G:K]}$. 
\end{proof}

\begin{remark} \label{rmk:abelian}
The ``Dedekind'' hypothesis is needed to ensure that $\res^G_K(\sigma_{G/H}) = \sigma_{K/H}^{[G:K]}$ in the last line of the previous proof. It is known \cite{Ded97} that a finite group is Dedekind if and only if it is abelian or 
isomorphic to $Q_8\x A$ where $A$ is a finite abelian group whose 2-component 
has the form $(\Z/2)^r$
for $r\geq 0$.

We remark that the conclusion of \cref{lem:RK-to-RG} holds for a more general class of nonabelian groups than those mentioned in the previous paragraph, but an explicit characterization is not known. For example, it can be checked that it holds for dihedral groups of order $2n$ for $n$ squarefree.
\end{remark}

\begin{lemma} \label{lem:PhiG-generated}
Assume $M$ is $H$-stabilizable and $M_{(H)}/G$ is connected for every $H\leq G$. For every $d$ there exists $N_d$ such that
$$ \pi_d(\Phi^G(H\underline{\Z}\sm C(M))) \isom P_G\cdot
\pi_d(\Phi^G(H\underline{\Z}\sm C_{\leq N_d}(M))). $$
Moreover, if $H_i(M_{(H)}/G)$ is a finitely generated abelian group for all $i\leq d$ and $(H)\leq G$, then
$\pi_d(\Phi^G(H\underline{\Z}\sm C_{\leq N_d}(M)))$ is a finitely generated abelian group.
\end{lemma}
\begin{proof}
Using the splitting in \cref{Cor:CGSDecomp} we have
\[
\pi_d(\Phi^G(H\underline{\Z}\sm C(M))) \isom \pi_d(\Phi^GH\underline{\Z}\sm C(M)^G)
\isom \bigoplus_{G\text{-sets }S} \pi_d(\Phi^G(H\underline{\Z})\sm C_S^G(M)).
\]
By Proposition \ref{Prop:Green}, there is a finite collection of nonnegative integers $n_i$ and finitely generated abelian groups $A_i$ such that 
\[ \pi_d(\Phi^G(H\underline{\Z})\sm C_S^G(M))\isom \bigoplus_{i} H_{d-n_i}(C_S^G(M); A_i). \]

For fixed $i$ and $H$, Theorem \ref{Thm:CGSHomStab} says that there exists
$N_{d,i,H}$ such that
$$H_{d-n_i}(C_S^G(M); A_i) = \sigma_{G/H}\cdot H_{d-n_i}(C_{S-[G/H]}^G(M);
A_i) $$
whenever $S$ contains at least $N_{d,i,H}$ copies of $G/H$. For fixed $i$, there
are finitely many $G$-sets $S$ such that $H_{d-n_i}(C_S^G(M);A_i)$ is not in the
image of $\sigma_{G/H}$ for some $H$. Let $N_{d,i}$ be the maximum cardinality
of the exceptional $G$-sets, and
let $N_d = \max_i\{ N_{d,i} \}$. Then for every $S$ and $n_i$, we have
\begin{align*}
H_{d-n_i}(C_S^G(M); A_i) &\subseteq P_G\cdot \bigoplus_{|S'| \leq N_d} H_{d-n_i}(C_{S'}^G(M); A_i).
\end{align*}
Summing over $i$, we obtain $\pi_d(\Phi^G(H\underline{\Z}\sm C^G_S(M)))\subseteq
P_G \cdot \pi_d(\Phi^G(H\underline{\Z}\sm C_{\leq N_d}(M)))$.

This shows the first statement. For the second statement, combine Corollary \ref{Cor:CGSDecomp} and Proposition \ref{Prop:MtoCnM}, which holds for homology with $A_i$-coefficients because $A_i$ is a finitely generated abelian group.
\end{proof}

\begin{proposition}\label{Prop:BredonHomStab}
Let $G$ be a Dedekind group. For $K\leq G$, let $i_K^*$ denote restriction from
$G$-spaces to $K$-spaces.
Assume that for all $H\leq K \leq G$, $M$ is
$H$-stabilizable, $(i_K^*M)_{(H)}/K$ is connected,
and $H_i(i_K^*(M)_{(H)}/K)$ is a finitely generated abelian group for all $i\leq d+1$.
Then $H_d^K(C(M);\underline{\Z})$ is finitely generated over $P_G$ for all $K\leq G$.
\end{proposition}
\begin{proof}
We will apply Lemma \ref{Lemma:Bredon-to-PhiG} with the Serre class of finitely generated $P_G$-modules. It suffices to check the hypothesis, namely that $\pi_i(\Phi^K(H\underline{\Z}\sm C(M)))$ is finitely generated over $P_G$ for all $K\leq G$ and $i\leq d+1$. 
We will apply Lemma \ref{lem:PhiG-generated} to $i^*_K(M)$. Checking the
hypotheses, first note that if $M$ is $H$-stabilizable as a $G$-manifold, then
$i^*_K(M)$ is $H$-stabilizable.
Thus $\pi_i(\Phi^K(H\underline{\Z}\sm C(i^*_K(M))))=\pi_i(\Phi^K(H\underline{\Z}\sm C(M)))$ is finitely generated over $P_K$. Now use Lemma \ref{lem:RK-to-RG} to get finite generation over $P_G$.
\end{proof}

\subsection{Simplifying the hypotheses of Proposition \ref{Prop:BredonHomStab}}
We will show (Lemma \ref{lem:fg-homology}) that the finite generation hypothesis in Proposition
\ref{Prop:BredonHomStab} follows from the assumption that $M$
the interior of a compact $G$-manifold $W$. Such finite type hypotheses are
common in the literature on this subject; see for example \cite{RW13} for the
non-equivariant version of our result. We also translate the connectedness hypothesis of Proposition
\ref{Prop:BredonHomStab} into a condition
which involves $M^H$ instead of $M_{(H)}$ (see Lemma
\ref{lem:hypothesis-translation}). The results are combined in \cref{thm:B}.
\begin{lemma} \label{lem:hypothesis-translation}
Suppose $G$ is Dedekind or a $p$-group, and let $W = W_GH$ denote the Weyl
group. If $M^H/W$ is connected
then $M_{(H)}/G$ (if nonempty) is
connected. Moreover,
$$ M_{(H)}/G \cong \Big(M^H \backslash \bigcup_{e < L \leq W} (M^H)^L\Big)/W. $$
\end{lemma}
\begin{proof}
Applying Lemma \ref{Lem:modW} to the $G$-manifold $M$ and the $W$-manifold
$M^H$, respectively, we have
\begin{align*}
M_{(H)}/G & \cong \Big(M^H \backslash \bigcup_{H<K\leq G} M^K\Big)\big/ W
\\ (M^H)_{(e)}/W & \cong \Big( M^H \backslash \bigcup_{e < L\leq W} (M^H)^L
\Big)\big/ W \cong \Big(M^H\backslash \bigcup_{H < K \leq N_G(H)} M^K\Big)\Big/W.
\end{align*}
Because of the group-theoretic conditions, $H < K \leq G$ implies $H < N_K(H)$, and
hence $M_{(H)}/G\cong (M^H)_{(e)}/W$.

If $M_{(H)}$ is nonempty, then $(e)$ is the principal orbit type of $M^H$ as a
$W$-manifold.
Then \cite[Theorem I.5.14]{tom-dieck-transformation-groups} applied to
$M^H$ says that $(M^H)_{(e)}/W \cong M_{(H)}/G$ is connected.
\end{proof}

The next lemma is used to prove Lemma \ref{lem:fg-homology}.

\begin{lemma}\label{lem:local-coeffs}
Suppose $M$ is $G$-manifold that is the interior of a compact $G$-manifold
with boundary. Let $H\leq G$. If $A$ is a $\pi_1(M^H)$-module that is finitely
generated as an abelian group, then the homology with local coefficients
$H_q(M^H;A)$ is a finitely generated abelian group for every $q$.
\end{lemma}
\begin{proof}
Let $Z$ be a compact $G$-manifold whose interior is $M$.
First, we observe that $M^H$ is the interior of $Z^H$;
this can be checked on Euclidean neighborhoods. Since $Z^H$ is a compact
manifold with boundary, it has a collar neighborhood, and this can be used to
produce a homotopy equivalence $M^H \hteq Z^H$. Since $\pi_1(M^H) = \pi_1(Z^H)$,
we may view $A$ as a sheaf of abelian
groups over $Z^H$, and consider an open cover $\mathscr{U}$ that trivializes
$A$. There is a refinement $\mathscr{U}'$ that is a good cover (see
\cite[Corollary 5.2]{bott-tu}) and since $Z^H$ is compact, we may take
$\mathscr{U}'$ to be finite. Since this is a good cover, homology with local
coefficients $H_q(Z^H;A)=H_q(M^H;A)$ agrees with \v{C}ech cohomology
$\check{H}_q(\mathscr{U}';A)$. The \v{C}ech complex
$$ \check{C}_q(\mathscr{U}';A) = \bigoplus_{(i_0,\dots,i_q)} A(U_{i_0}\ints
\dots \ints U_{i_q}) = \bigoplus_{(i_0,\dots,i_q)}A $$
is a finitely generated abelian group in each degree, and hence so is its homology.
\end{proof}

\begin{lemma} \label{lem:fg-homology}
Assume $G$ is abelian, and $M$ is the interior of a compact $G$-manifold.
Then $H_d(M_{(H)}/G; \Z)$ is finitely generated for all $d\in \Z,H\leq G$.
\end{lemma}
\begin{proof}
Let $W = W_GH$ be the Weyl group.
Lemma \ref{lem:hypothesis-translation} says that $M_{(H)}/G \cong X/W$ where
\begin{equation}\label{eq:bigcupX} X = M^H \backslash \bigcup_{e < L\leq W} (M^H)^L.
\end{equation}
By Lemma \ref{lem:homology-orbits-finite}, it suffices to show that
$H_*(X;\Z)$ is finitely generated in each
degree.

If $q:G\to W = G/H$ is the quotient, then $(M^H)^L = M^{q^{-1}(L)}$.
Abusing notation, we write $M^L = (M^H)^L$; also let
$U_L = M^H \backslash M^L$.
Since $M^L$ is a submanifold, there is a normal bundle $N(M^L)$ which is open,
and $N(M^L) \cup U_L \cong M^H$.
First we will show $H_d(U_L)$ is finite using a Mayer--Vietoris sequence.

The intersection term $N(M^L)\cap U_L$ is homotopy equivalent to the sphere bundle $S(N(M^L))$. There
is a Serre spectral sequence associated to the fibration $S^n\to S(N(M^L))\to
M^L$,
$$ E_2^{*,*} = H_*(M^L; H_*(S^n; \Z))\Rightarrow H_*(S(N(M^L));\Z) $$
where $\pi_1(M^L)$ may act nontrivially on $H_*(S^n;\Z)$. By Lemma
\ref{lem:local-coeffs}, the $E_2$ term is finitely generated in each degree, and
hence so is the target.

Now consider the Mayer--Vietoris sequence
$$ \dots\to H_{d+1}(M^H)\to H_d(S(N(M^L))) \to H_d(N(M^L)) \oplus H_d(U_L)\to
\dots. $$
The first term is finitely generated by Lemma \ref{lem:local-coeffs} and above we showed the second is
as well. Hence the third term is finitely generated; in particular so is $H_d(U_L)$.

Since $G$ is abelian, write $W = \prod_i \Z/p_i^{r_i}$.	In the union \eqref{eq:bigcupX}, it
suffices to restrict $L$ to products of elementary abelian subgroups
$\Z/p_i\cong \langle p_i^{r_i-1}\rangle\leq \Z/p_i^{r_i}$, since any nontrivial
subgroup $H\leq \Z/p_i^{r_i}$
contains this $\Z/p_i$. If $L\x L' \leq W$, we have $(M^H)^L \ints (M^H)^{L'} =
(M^H)^{L\x L'}$. More generally,
given $L = \prod_i L_i$ and $L' = \prod_i L'_i$, note
that $LL'\leq W$ is a product of all $L_i$ and $L'_i$ factors, with duplicates
removed. Thus we have
$$ U_L \cup U_{L'} = M^H \backslash ((M^H)^L \cap (M^H)^{L'}) = M^H \backslash
(M^H)^{LL'}. $$
We will show that 
$$X = \bigcap_\attop{L = L_1 \x \dots \x L_n\\L_i\text{ elem.ab.}} U_L $$
has finitely generated homology, by induction on the number of sets intersected.
The base case was shown above. Suppose any intersection of $n$ such sets has
finitely generated homology, and consider $U_{L_1}\ints \dots \ints U_{L_n}
\ints U_L$. There is a Mayer-Vietoris sequence
$$ \dots \to H_{d+1}((\cap_{i=1}^n U_{L_i})\cup U_L) \to H_d((\cap_{i=1}^n
U_{L_i})\cap U_L) \to H_d(\cap_{i=1}^n U_{L_i}) \oplus H_d(U_L) \to \dots. $$
The inductive hypothesis implies that the third term is finitely generated, so
in order to show the second term is finitely generated, it suffices to show the first
term is finitely generated. We have
$$ \Big(\bigcap_{i=1}^n U_{L_i}\Big) \cup U_L = \bigcap_{i=1}^n (U_{L_i} \cup
U_L) = \bigcap_{i=1}^n U_{L_i L} $$
and the inductive hypothesis shows this has finitely generated homology.
\end{proof}

\begin{theorem}\label{thm:B}
Let $G$ be an abelian group and let $M$ be a $G$-manifold that is the interior
of a compact $G$-manifold with boundary. Assume that for all $H\leq G$, $M$ is
$H$-stabilizable and $M^H$ is connected.
Then $H_d^H(C(M);\underline{\Z})$ is finitely generated over $P_G$ for all $H\leq G$.
\end{theorem}
\begin{proof}
Combine \cref{Prop:BredonHomStab}, \cref{lem:hypothesis-translation}, and
\cref{lem:fg-homology}.
\end{proof}

\begin{example} \label{ex:rho}
If $G$ is abelian and $V = n\rho_G$ is a sum of regular representations, the
hypotheses of Theorem \ref{thm:B} are satisfied. One could also check the
hypotheses of Proposition \ref{Prop:BredonHomStab} directly without using the
results of this subsection: by Lemma
\ref{Lem:modW} the space $V_{(H)}$ arises from the complement of hyperplanes in
$n\rho_G$. This complement is Alexander dual to a union of spheres, one for each
hyperplane.
\end{example}

\section{Generalizations and variants}\label{Sec:Gen}

Throughout this section, $G$ denotes a finite group in which every subgroup is normal. We have established stability for \emph{$\z$-graded} Bredon homology \emph{groups}. In this section, we discuss natural extensions of our results to
\begin{enumerate}
\item $RO(G)$-graded Bredon homology groups, and
\item $\z$-graded Bredon homology Mackey functors.
\end{enumerate}
Here, $RO(G)$ is the (group completion) of the ring of isomorphism classes of real orthogonal representations of $G$. 

\subsection{$RO(G)$-graded Bredon homological stability}

Let $V \in RO(G)$. We write $-V$ for the additive inverse of $V$ in $RO(G)$ and write $S^V$ for the representation sphere associated to $V$. 

\begin{definition}
Let $V \in RO(G)$. The $V$-th Bredon homology group of a $G$-space $X$ with coefficients in the Mackey functor $\uM$ is the group
$$\pi_0^G\left( S^{-V} \wedge H\uM \wedge X \right) \cong H_0^G( S^{-V} \wedge X; \uM).$$
\end{definition}

The following is well-known:

\begin{lemma}\label{Lem:RepSpheres}
If $V \in RO(G)$, then we have 
$$(S^V)^e = \Phi^e(S^V) = S^{|V|}, \quad  \Phi^G(S^V) = S^{|V^G|}.$$
\end{lemma}

Recall that $P_G = \Z[(\sigma_{G/H})_* : (H)\leq G]$.
\begin{theorem}\label{Thm:ROG}
Let $V \in RO(G)$ and let $M$ be a $G$-manifold satisfying the hypotheses of
\cref{thm:B}. Then the $P_G$-module $H_V^G(C(M);\mfz)$ is finitely generated. 
\end{theorem}

\begin{proof}
Applying \cref{Lemma:Bredon-to-PhiG} with $X = S^{-V} \wedge C(M)$ and $\cc$ the Serre class of finitely generated $P_G$-modules, it suffices to show that $\pi_q \Phi^K(H\mfz \wedge S^{-V} \wedge C(M))$ is in $\cc$ for all $K \leq G$ and all $q \in \z$. \cref{Lem:RepSpheres} implies that
$$\pi_q \Phi^K \left(H\mfz \wedge S^{-V} \wedge C(M) \right) \cong \pi_{q+|V^K|} \Phi^K \left(H\mfz \wedge C(M) \right),$$
and this is finitely generated by combining the same ingredients as the proof of
Proposition \ref{Prop:BredonHomStab} and Theorem \ref{thm:B}.
\end{proof}

\subsection{Mackey functor stability}

If $L \leq G$ is a subgroup, we may view any $G$-manifold $M$ as an $L$-manifold by restriction. If the $G$-manifold $M$ satisfies the hypotheses of \cref{MT:Stability}, then its underlying $L$-manifold clearly also satisfies the conditions of \cref{MT:Stability} (with $L$ in place of $G$) and we obtain the following:

\begin{corollary}
Suppose $M$ satisfies the hypotheses of \cref{thm:B}.
Then $H_d^L(C(M);\mfz)$ is finitely generated over $P_L$ for all integers $d \geq 0$.  
\end{corollary}

Recall that Bredon homology is actually a Mackey functor, i.e., an additive functor from the Burnside category of $G$ to abelian groups. The previous corollary implies that the value of the Mackey functor
$$H_d^{(-)}(C(M); \mfz)$$
stabilizes when evaluated on any object in the Burnside category (i.e., any finite $G$-set). We can encode this levelwise stability, and more, with the following notion:

\begin{definition}
A \emph{module} over a Green functor $\uR$ is a Mackey functor $\uM$ equipped with an action map $\uR \boxtimes \uM \to \uM$ satisfying the usual associativity and unitality conditions. 

An $\uR$-module $\uM$ is \emph{finitely generated} if there is a surjective $\uR$-linear map $\uR\{x_T\} \to \uM$, where $\uR\{x_T\}$ is the free $\uR$-module on a generator $x_T$ with $T$ a finite $G$-set. 
\end{definition}

According to \cite[Lem. 2.17]{Shu10}, a map $\uR \boxtimes \uM \to \uM$ determines, and is determined by, its  \emph{Dress pairing}: a collection of maps $\uR(S) \otimes \uM(S) \to \uM(S)$ for $S$ in the Burnside category of $G$ satisfying certain compatibility axioms.  Using Dress pairings, we can relate levelwise finite generation to Mackey finite generation:

\begin{lemma}\label{Lem:FGLevelwise}
An $\uR$-module $\uM$ is finitely generated if and only if $\uM(S)$ is finitely generated as a $\uR(S)$-module for each finite $G$-set $S$.
\end{lemma}

\begin{proof}
Surjectivity for Mackey functors is checked levelwise. 
\end{proof}

To express stability in terms of Mackey functors, we need a Green functor to encode the action of the equivariant stabilization maps for each subgroup of $G$. This is accomplished through the following:

\begin{definition}
For each finite group $G$, let $\uP_G$ be the constant Green functor on $P_G=\z[(\sigma_{G/H})_*: (H) \leq G]$. (\emph{A priori}, $\uP_G$ is only a Mackey functor. It is a Green functor since it is the fixed point functor of a commutative ring with trivial action.) We define an action of $\uP_G$ on $H_*^{(-)}(C(M); \mfz)$ via Dress pairing: the map
$$\uP_G(G/H) \otimes H_*^H(C(M);\mfz) \cong P_G\otimes H^H_*(C(M);\mfz) \to H_*^H(C(M);\mfz)$$
comes from the restriction map $\res^G_H:P_G\to P_H$ as defined in Section \ref{SS:BredonStab}.
\end{definition}

In particular, if $H_*^H(C(M);\mfz)$ is finitely generated over $\uP_H(H/H)$ for every subgroup $H \leq G$, and if $\uP_H(H/H)$ is finitely generated as a $\uP_G(G/H)$-module via the restriction map (e.g., if every subgroup of $G$ is normal), then $H_*^H(C(M);\mfz)$ is finitely generated as a $\uP_G(G/H)$-module. Combined with \cref{Lem:FGLevelwise}, this proves:

\begin{theorem}\label{Thm:MF}
Let $M$ be as in \cref{thm:B}. Then $H_d^{(-)}(C(M); \mfz)$ is finitely generated over $\uP_G$ for all integers $d \geq 0$.  
\end{theorem}

\bibliographystyle{alpha}
\bibliography{master}

\end{document}